%% file: arxiv.tex
\documentclass[12pt]{article}

\usepackage[english]{babel}
\usepackage{amsmath, amsfonts, amssymb, amsthm}
\usepackage{bbm}
\usepackage{url}
\usepackage{hyperref}
\usepackage[misc,geometry]{ifsym} 
\usepackage{enumitem}
\usepackage{fullpage}
\usepackage{mathtools}
\usepackage[T2A]{fontenc}
\usepackage[utf8]{inputenc}
\usepackage{graphicx}
\usepackage[dvipsnames]{xcolor}
\usepackage{subcaption}

\usepackage{array, tabularx, caption, boldline}
\usepackage{natbib}

\usepackage{graphicx}

\usepackage{algorithm}
\usepackage{algorithmic}
\usepackage{caption}

\usepackage{cellspace}
\usepackage{collectbox}

\newcommand*{\R}{{\mathbb R}}

\newcommand*{\e}{\varepsilon}
\newcommand*{\la}{\langle}
\newcommand*{\ra}{\rangle}

\newcommand*{\const}{\mathrm{const}}

\def\C{\mathcal C}

\def\W{\mathcal W}

\newcommand*{\one}{\mathbbm{1}}

\def\bld{\boldsymbol}
\def\Blm{\bld{\lambda}}

\newcommand*{\Bmu}{\bld\mu}
\newcommand*{\Bpi}{\bld\pi}

\def\p{\mathtt{p}}

\def\q{\mathtt{q}}
\renewcommand{\u}{\mathtt{u}}
\renewcommand{\v}{\mathtt{v}}

\def\lm{\lambda}

\newcommand*{\norm}[1]{\left\lVert#1\right\rVert}

\DeclareMathOperator*{\argmin}{argmin}

\DeclareMathOperator{\spn}{span}
\DeclareMathOperator{\diag}{diag}
\DeclareMathOperator{\nnz}{nnz}

{
\theoremstyle{remark}
\newtheorem{Rem}{Remark}

}
{
\theoremstyle{definition}

}

\newtheorem{theorem}{Theorem}
\newtheorem{lemma}{Lemma}

\def\ddm#1{{\color{black}#1}} 
\def\pd#1{{\color{black}#1}} 
\def\ak#1{{\color{black}#1}} 
\def\ag#1{{\color{black}#1}} 
\def\gav#1{{\color{black}#1}} 

\title{On the Complexity of Approximating Wasserstein Barycenter} 
\author{
Alexey Kroshnin
\thanks{Institute for Information Transmission Problems RAS, Moscow; National Research University Higher School of Economics, Moscow, akroshnin@hse.ru}
\and
Darina Dvinskikh
\thanks{Weierstrass Institute for Applied Analysis and Stochastics, Berlin; Institute for Information Transmission Problems RAS, Moscow, darina.dvinskikh@wias-berlin.de}
\and
Pavel Dvurechensky
\thanks{Weierstrass Institute for Applied Analysis and Stochastics, Berlin; Institute for Information Transmission Problems RAS, Moscow, pavel.dvurechensky@wias-berlin.de}
\and
Alexander Gasnikov
\thanks{Moscow Institute of Physics and Technology, Moscow; Institute for Information Transmission Problems RAS, Moscow, gasnikov@yandex.ru}
\and
Nazarii Tupitsa
\thanks{Institute for Information Transmission Problems RAS, Moscow, tupitsa@phystech.edu}
\and
C\'{e}sar A. Uribe 
\thanks{Massachusetts Institute of Technology, Cambridge, cauribe@mit.edu} 
}							

\date{\today}

\begin{document}

\maketitle

\begin{abstract}
We study the complexity of approximating Wassertein barycenter of $m$ discrete measures, or histograms of size $n$ by contrasting two alternative approaches, both using entropic regularization. The first approach is based on the Iterative Bregman Projections (IBP) algorithm for which our novel analysis gives a complexity bound proportional to $\frac{mn^2}{\e^2}$ to approximate the original non-regularized barycenter. 
 Using an alternative accelerated-gradient-descent-based approach, we obtain a complexity proportional to
$\frac{mn^{2\gav{.5}}}{\e} $. 
As a byproduct, we show that the regularization parameter in both approaches has to be proportional to $\e$, which causes instability of both algorithms when the desired accuracy is high. To overcome this issue, we propose a novel proximal-IBP algorithm, which can be seen as a proximal gradient method, which uses IBP on each iteration to make a proximal step.
We also consider the question of scalability of these algorithms using approaches from distributed optimization and show that the first algorithm can be implemented in a centralized distributed setting (master/slave), while the second one is amenable to a more general decentralized distributed setting with an arbitrary network topology.

\end{abstract}


\section*{Introduction}
Optimal transport (OT) \cite{Monge1781,kantorovich1942translocation} is currently generating an increasing attraction in statistics, machine learning and optimization communities. Statistical procedures based on optimal transport are available \cite{bigot2012consistent,barrio2015statistical,ebert2017construction,le2017existence} as well as many applications in different areas of machine learning including unsupervised learning \cite{arjovsky2017wasserstein,bigot2017geodesic}, semi-supervised learning \cite{solomon2014wasserstein}, clustering \cite{ho17multilevel}, text classification \cite{kusner2015from}. 
Optimal transport distances lead to the concept of Wasserstein barycenter, which allows to define a mean of a set of complex objects, e.g. images, preserving their geometric structure \cite{cuturi2014fast}.
In this paper we focus on the computational aspects of optimal transport, namely on approximating Wasserstein barycenter of a set of histograms.

Starting with \cite{altschuler2017near-linear}, several groups of authors addressed the question of Wasserstein distance approximation complexity \cite{chakrabarty2018better,dvurechensky2018computational,blanchet2018towards,lin2019efficient}. Implementable schemes are based on Sinkhorn's algorithm, which was first applied to OT in \cite{cuturi2013sinkhorn}, and accelerated gradient descent proposed as an alternative in \cite{dvurechensky2018computational}.
Much less is known about the complexity of approximating Wasserstein \textit{barycenter}. The works \cite{staib2017parallel,uribe2018distributed,dvurechensky2018decentralize}, are in some sense close, but do not provide an explicit answer. 
Following \cite{genevay2016stochastic}, \cite{staib2017parallel} use stochastic gradient descent and estimate the convergence rate of their algorithm. From their rate, one can obtain the iteration complexity $\frac{\kappa R^2}{\e^2}$ to achieve accuracy $\e$ in approximation of the barycenter, where $\kappa$ is some constant depending on the problem data, i.e. transportation cost matrices, $R$ is some distance characterizing the solution of the dual problem. \cite{dvurechensky2018decentralize} consider regularized barycenter, but do not show, how to choose the regularization parameter to achieve $\e$-accuracy.  

Following \cite{dvurechensky2018computational}, we study two alternative approaches for approximating Wasserstein barycenter based on entropic regularization \cite{cuturi2013sinkhorn}. The first approach is based on Iterative Bregman Projection (IBP) algorithm \cite{benamou2015iterative}, which can be considered as a general alternating projections algorithm and also as a generalization of the Sinkhorn's algorithm \cite{sinkhorn1974diagonal}. The second approach is based on constructing a dual problem and solving it by primal-dual accelerated gradient descent. For both approaches, we show, how the regularization parameter should be chosen in order to approximate the original, non-regularized barycenter. A kind of mixed approach, namely an accelerated Iterative Bregman Projection algorithm was proposed in \cite{guminov2019acceleratedAM}. 

We also address the question of scalability of computations in the Big Data regime, i.e. the size of histograms $n$ and the number of histograms $m$ are large. In this case the dataset of $n$ histograms can be distributedly produced or stored in a network of agents/sensors/computers with the network structure given by an arbitrary connected graph. In a special case of centralized architecture, i.e. if there is a central "master" node surrounded by "slave" nodes, parallel algorithms such as \cite{staib2017parallel} can be applied. In a more general setup of arbitrary networks it makes sense to use decentralized distributed algorithms in the spirit of distributed optimization algorithms \cite{Nedic2017cc, Scaman2017Optimal}.


\subsection*{Related Work}
It is very hard to cover all the increasing stream of works on OT and we mention these books \cite{villani2008optimal,santambrogio2015optimal,peyre2018computational} as a starting point and the references therein. Approximation of Wasserstein barycenter was considered in
\cite{cuturi2014fast,bonneel2015sliced,benamou2015iterative,staib2017parallel,puccetti2018computation,claici2018stochastic,uribe2018distributed,dvurechensky2018decentralize} using different techniques as Sinkhorn-type algorithm, first-order methods, Newton-type methods.
Considering the primal-dual approach based on accelerated gradient descent, our paper shares some similarities with \cite{cuturi2016smoothed} with the main difference that we are focused on complexity and scalability of computations and explicitly analyze the algorithm applied to the dual problem.

There is a vast amount of literature on accelerated gradient descent with the canonical reference being \cite{nesterov1983method}. Primal-dual extensions can be found in \cite{lan2011primal-dual,tran-dinh2015smooth,yurtsever2015universal,chernov2016fast,dvurechensky2016primal-dual,dvurechensky2017adaptive,anikin2017dual,nesterov2018primal-dual,lin2019efficient}. We are focused on the extensions amenable to the decentralized distributed optimization, so that these algorithms can be scaled for large problems.

Distributed optimization algorithms were considered by many authors with the classical reference being \cite{bertsekas1989parallel}. Initial algorithms, such as Distributed Gradient Descent~\cite{nedic2009distributed}, were relatively slow compared with their centralized counterparts. However, recent work has made significant advances towards a better understanding of the optimal rates of such algorithms and their explicit dependencies to the function and network parameters~\cite{lan2017communication,Scaman2017Optimal,uribe2018dual}. These approaches has been extended to other scenarios such as time-varying graphs~\cite{rogozin2018optimal,maros2018panda,wu2017fenchel}. The distributed setup is particularly interesting for machine learning applications on the big data regime, where the number of data points and the dimensionality is large, due to its flexibility to handle intrinsically distributed storage and limited communication, as well as privacy constraints~\cite{he2018cola,wai2018sucag}.

\subsection*{Our contributions}

\begin{itemize}
    \item We consider $\gamma$-regularized Wasserstein barycenter problem and obtain complexity bounds for finding an approximation to the regularized barycenter by two algorithms. The first one is Iterative Bregman Projections algorithm \cite{benamou2015iterative}, for which we prove complexity proportional to $\frac{1}{\gamma \e}$
     to achieve accuracy $\e$. The second one is based on accelerated gradient descent (AGD) and has complexity proportional to $\sqrt{\frac{n}{\gamma \e}}$. The benefit of the second algorithm is that it is better scalable and can be implemented in the decentralized distributed optimization setting over an arbitrary network.
    \item We show, how to choose the regularization parameter in order to find an $\e$-approximation for the non-regularized Wasserstein barycenter and find the resulting complexity for IBP to be proportional to $\frac{m n^2}{\e^2}$ and for AGD to be proportional to $\frac{m n^{2.5}}{\e}$.
    \item As we can see from the complexity bounds for IBP and AGD, they depend on the regularization parameter $\gamma$ quite badly regarding that this parameter has to be small. To overcome this drawback we propose a proximal-IBP method, which can be considered as a proximal method using IBP on each iteration to find the next iterate.
\end{itemize}

\section{Problem Statement and Preliminaries}
\label{S:prel}

\subsection{Notation}
We define the probability simplex as 
$S_n(1) = \{q \in \R_+^n \mid \sum_{i=1}^n q_i = 1\}$. 
Given two discrete measures $p$ and $q$ from $S_n(1)$ we introduce the set of their coupling measures as
\begin{equation*}
    \Pi(p,q) = \{ \pi \in \R_+^{n \times n}: \pi \one = p, \pi^T \one = q\}.
\end{equation*}
For coupling measure $\pi \in \R_+^{n \times n}$ we denote the negative entropy (up to an additive constant) as
\[H(\pi) := \sum^n_{i, j=1} \pi_{i j} \left(\ln \pi_{i j} - 1\right) = \la \pi, \ln \pi - \one \one^T \ra.\]
Here and further by $\ln(A)$ ($\exp(A)$) we denote the element-wise logarithm (exponent) of matrix or vector $A$, and $\la A, B \ra := \sum_{i, j = 1}^n A_{i j} B_{i j}$ for any $A, B \in \R^{n \times n}$. For two matrices $A$ $B$ we also define element-wise multiplication and element-wise division as $ A \odot B$ and $\frac{A}{B}$ respectively.
Kullback-Leibler divergence for measures $\pi, \pi' \in \R_+^{n \times n}$ is defined as the Bregman divergence associated with $H(\cdot)$:
\begin{equation*}
    KL(\pi | \pi') 
    := \sum_{i, j = 1}^n \left(\pi_{i j} \ln \left(\frac{\pi_{i j}}{\pi'_{i j}}\right) - \pi_{i j} + \pi'_{i j}\right)
    = \la \pi, \ln \pi - \ln \pi' \ra + \la \pi' - \pi, \one \one^T \ra.
\end{equation*}
We also define a symmetric cost matrix
$C \in \R_+^{n \times n}$, which element $c_{i j}$ corresponds to the cost of moving bin $i$ to bin $j$. $\norm{C}_\infty$ denotes the maximal element of this matrix.

 We refer to $\lm_{\max}(W)$ as the maximum eigenvalue of a symmetric matrix $W$, and $\lm^+_{\min}(W)$ as the minimal non-zero eigenvalue, and define the condition number of matrix $W$ as $\chi(W)$. We use symbol $\one$ as a column of ones.


\subsection{Wasserstein barycenters and entropic regularization}

Given two probability measures $p,q \in S_n(1)$ and a cost matrix $C \in \R^{n \times n}$ we define optimal transportation distance between them as
\begin{equation}\label{eq:OT}
    \W(p, q) := \min_{\pi \in \Pi(p,q)} \la \pi, C \ra.
\end{equation}

For a given set of probability measures $\{p_1, \dots, p_m\}$ and cost matrices $C_1, \dots, C_m \in \R_+^{n \times n}$ we define their weighted barycenter with weights $w \in S_n(1)$ as a solution of the following convex optimization problem:
\begin{equation}\label{prob:unreg_bary}
    \min_{q \in S_n(1)} \sum_{l=1}^m w_l \W(p_l, q)\gav{,}
\end{equation}
\gav{where $w_l \ge 0$, $l = 1,...,m$ and}
\gav{$$\sum_{l=1}^m w_l = 1.$$}

We use $c$ to denote $\max_{l=1,...,m} \norm{C_l}_{\infty}$.
Using entropic regularization proposed in \cite{cuturi2013sinkhorn} we define regularized OT-distance for $\gamma \geq 0$:
\begin{equation}\label{eq:reg_OT}
    \W_\gamma(p,q) := \min_{\pi \in  \Pi(p,q)} \left\{ \la \pi, C \ra + \gamma H(\pi) \right\}.
\end{equation}
Respectively, one can consider regularized barycenter that is a solution of the following problem:
\begin{equation}\label{prob:reg_bary}
    \min_{q \in S_n(1)} \sum_{l=1}^m w_l \W_\gamma(p_l, q).
\end{equation}

\section{Complexity of WB by Iterative Bregman Projections}
\label{S:IBP_compl}



\input{IBP_short}


 

\subsection{Proximal IBP for Wasserstein barycenter problem}

As we see from Theorems \ref{thm:complexity_IBP} and \ref{thm:complexity_IBP_to_bary}, to obtain an $\e$-approximation of the non-regularized barycenter, the regularization parameter $\gamma$ should be chosen proportional to small $\e$ and the complexity of the IBP is inversely proportional to $\gamma$, which leads to large working time and instability issues. To overcome this obstacle we propose a novel proximal-IBP algorithm. It is inspired by proximal point algorithm with general Bregman divergence $V(x,y)$ \cite{chen1993convergence}. The idea of this algorithm for minimization of a function $f(x)$ is to perform steps $x_{k+1} = \mathbf{prox}(x_k) =  \arg \min_{x \in Q} \{f(x) + \gamma V(x,x_k)\}$. We use the KL-divergence as the Bregman divergence since in this case the proximal step leads to a similar problem to the entropic-regularized WB \eqref{prob:reg_bary}.
Given the sets $\C_1,\C_2$ defined in \eqref{eq:C1C2Def}, we define proximal operator $\mathbf{prox}: \C_1 \cap \C_2 \rightarrow \C_1 \cap \C_2$ for  function $\sum_{l=1}^m \gav{w_l}W_{\gamma}(p_l,q_l)$ as follows\\
   \ddm{\begin{align}
 \mathbf{prox}(\Bpi^k) &=  \arg\min_{ \Bpi \in \C_1 \cap \C_2} \sum_{l=1}^m {w_l}\left\{\la \C_l, \pi_l  \ra + \gamma KL(\Bpi| \Bpi^k)\right\}\notag\\
  &   =\arg\min_{ \Bpi \in \C_1 \cap \C_2}  \sum_{i,j=1}^n\sum_{l=1}^m {w_l}\left\{ [\C_l]_{ij}\cdot [\pi_l]_{ij}   +  \gamma \left([\pi_{l}]_{ij} \ln \left(\frac{[\pi_{l}]_{ij}}{[\pi^k_{l}]_{ij}}\right) - [\pi_{l}]_{ij} + [\pi^k_{l}]_{ij}\right)\right\} \notag \\ 
    &= \arg\min_{\Bpi \in \C_1 \cap \C_2}  \sum_{l=1}^m {w_l} KL\left(\pi_l| \pi_l^k \odot \exp\left(-\frac{C_l}{\gamma}\right)\right), \notag
\end{align}}
The proximal gradient method has the following form
\begin{align}
    \Bpi^{k+1} = \mathbf{prox}(\Bpi^k).
\end{align}
Then we use Iterative Bregman Projection for finding the barycenter.

For convenience let's determine after Dual$\_$IBP$(\u^0,\v^0,\{K_l\},M)$ the following objects $\u^M$, $\v^M$; $\{B_l(u_l^M,v_l^M)\}$; $\bar{q}^M$ obtained after $M$ iterations of Algorithm~\ref{Alg:dual_IBP} applied with starting points $\u^0$, $\v^0$ and set of matrices $\{K_l\}$.

\begin{algorithm}[H]
\caption{Finding Wasserstein barycenter by proximal IBP}
\label{Alg:ProxIBP}
\begin{algorithmic}[1]
    \REQUIRE $M$, $N$~--- numbers of \gav{internal and external} iterations \pd{respectively}; $\u^{0,0} = 0$ $\v^{0,0} = 0$; starting transport plans $\{\pi_l^0\}_{l=1}^m: \pi_l^0 = \exp\left(-\frac{C_l}{\gamma}\right)$ $  ~\forall l=1, \dots, m$
    \STATE For each agent $l \in V$:
        \FOR{$k=0,\dots, N-1$}
            \STATE $K_l^k := \pi_l^k \odot \exp(-\frac{C_l}{\gamma}) $
             \STATE \pd{Run Algorithm \ref{Alg:dual_IBP} for $M$ iterations with starting point $\u^{0,k},\v^{0,k}$ and kernels $\{K_l^k\}$ instead of $\{K_l\}$. }
            \STATE \pd{Set  $u_l^{0,k+1} = u_l^M$, $v_l^{0,k+1} = v_l^M$, where the sequence of dual variables $(u_l^t,v_l^t)$ is generated by Algorithm \ref{Alg:dual_IBP}.}  
            \STATE \pd{Set $\pi_l^{k+1} = B_l(u_l^M,v_l^M)$.}
        \ENDFOR
    \ENSURE  \pd{$\bar{q}^M$, generated on the last inner iteration of Algorithm \ref{Alg:dual_IBP} on the last outer iteration.}
\end{algorithmic}
\end{algorithm} 

We underline that in this setting, there is no need to choose  $\gamma$  to be small as it prescribed by Theorem \ref{Th:gamma}.
Algorithm \ref{Alg:ProxIBP} has two loops: external loop of proximal gradient step and internal loop of computing the next iterate $\Bpi^k$ by IBP and as a byproduct an approximation $q^k$ to the barycenter. 
The number of external iterations is proportional to $\frac{\gamma}{\e}$, see \cite{chen1993convergence}, and the complexity of internal loop is proportional to \ag{ $\min \left\{\frac{\gav{\const}}{\gamma \tilde{\e}}, \exp\left(\frac{\const}{\gamma}\right) \ln\left(\frac{1}{\tilde{\e}}\right)\right\}$, where $\tilde{\e} = O(\gav{\epsilon^2/(mn^3)})$ is a required precision for inner problem \cite{stonyakin2019gradient}. The first estimate directly follows from IBP complexity (see Theorem \ref{thm:complexity_IBP})} 
\ak{and the second estimate (analogous to \cite{franklin1989scaling} for Sinkhorn's algorithm) can be obtained using strong convexity of $f(\u, \v)$. Namely, one can show that
(cf. Lemma~1 in \cite{dvurechensky2018computational})
\[
\max_i [u_l^t]_i - \min_i [u_l^t]_i \le \ln(\max_i [p_l]_i) - \ln(\min_i [p_l]_i) + \frac{\norm{C_l}_\infty}{\gamma}
\]
and thus for any $l$, $t$ it holds
\[
0 \le - [\ln B_l(u_l^t, v_l^t)]_{i j} \le 4 \frac{\max_l \norm{C_l}_\infty}{\gamma} + 2 \ln n - \ln \min_{i} [p_l]_i \quad \forall i, j.
\]
Moreover, this bound holds on the convex hull of $\{(\u^t, \v^t)\}_{t \ge 0}$. Therefore
\[
\lm_{min}^+(\nabla_{(u_l, v_l)}^2 f(\u^t, \v^t)\gav{)} \ge w_l \frac{\min_i [p_l]_i}{n^2} \exp\left(- 4 \frac{\max_l \norm{C_l}_\infty}{\gamma}\right).
\]
This estimate implies linear convergence and thus the following complexity bound for IBP:
\[
O\left(\frac{n^2}{\min_{l, i} [p_l]_i} \exp\left(\frac{4 \max_l \norm{C_l}_\infty}{\gamma}\right) \ln\left(\frac{1}{\tilde{\e}}\right)\right)
= O\left(\exp\left(\frac{\const}{\gamma}\right) \ln\left(\frac{1}{\tilde{\e}}\right)\right).\]
}
 \ag{Note that 
 \gav{since in practice $\const$ is not too big,} th\gav{e}n setting \gav{$\gamma = \tilde{O}(\const)$} one can \gav{typically} improve the complexity of IBP by KL-proximal envelope as is follows: $1/\e^2 \to 1/\e$. The last bound (based on the similar reasons as in \cite{blanchet2018towards}) seems to be \gav{unimprovable} for this problem (see also item \ref{S:APDAGD_compl}). In practice one should try to find $\const$ by restart procedure on $\gamma$ on the first external loop iteration. That is, we start with large enough $\gamma$ and solve internal problem by IBP, then put $\gamma : = \gamma / 2$ and solve internal problem (with the same precision $\tilde{\epsilon}$) once again. We stop this repeating procedure at the moment when the complexity of internal problem growth significantly. This moment allows us to detect the moment of $\gamma = \widetilde{O}(\const).$ On the next external iterations one may use this $\gamma$.}
 
 \gav{Numerical experiments and more accurate theoretical analysis can be found in the follow-up paper \cite{stonyakin2019gradient}.}

Algorithm \ref{Alg:ProxIBP} can be implemented in centralized distributed setting in the same way as Algorithm~\ref{Alg:dual_IBP}, see Section \ref{ConvergenceIBP}. 


\section{Complexity by Primal-Dual Accelerated Gradient Descent}
\label{S:APDAGD_compl}

In this subsection we consider Primal-Dual Accelerated Gradient Descent for approximating Wasserstein barycenter. First, we consider regularized barycenter, construct a dual problem to \eqref{prob:reg_bary} and apply primal-dual accelerated gradient descent to solve it and approximate the regularized berycenter. Our dual problem is constructed via a matrix $W$, which can be quite general. We explain how the choice of this matrix is connected to distributed optimization and allows to implement the algorithm in the decentralized distributed setting. Then, we show, how the regularization parameter should be chosen in order to obtain an $\e$-approximation for the non-regularized Wasserstein barycenter, and estimate the complexity of the resulting algorithm. \ag{The proposed algorithms can be implemented in a decentralized distributed manner such that each node fulfils $\widetilde{O}(n^{2.5}/\e)$ arithmetic operations and the number of communication rounds is $\widetilde{O}(\sqrt{n}/\e)$ .}

\subsection{Consensus view on Wasserstein barycenter problem}
We rewrite the problem  \eqref{prob:reg_bary} in an equivalent way as 
\begin{align}\label{eq:Primal}
    \min_{\substack{q_1,\dots, q_m \in S_n(1) \\q_1=\cdots=q_m}} 
    W_{\gamma}(\p, \q):=
    \sum\limits_{l=1}^{m} w_l\W_{\gamma(l)}(p_l, q_l),
\end{align}
where $\p = [p_1,\dots, p_m]^T$ and $\q = [q_1, \dots, q_m]^T$,  we also use different regularizer $\gamma_l = \gamma(l)$ for $l$-th Wasserstein distance.
Next we write a dual problem by dualizing equality constraints $q_1=\cdots=q_m$. This can be done in many different ways and we do it by introducing a matrix $\bar W\in \R^{n\times n}$ which is a symmetric positive semi-definite matrix s.t. $\rm Ker(\bar W) = \spn(\one)$.  
Then, defining $W = \bar W \otimes I_n$ and using the fact $q_1=\cdots=q_m \Longleftrightarrow \sqrt{W}\q =0$, we equivalently rewrite problem \eqref{eq:Primal} as
\begin{align}\label{consensus_problem2}
    \max_{\substack{q_1,\dots, q_m \in S_n(1), \\ \sqrt{W} \q = 0 }} ~ - \sum\limits_{l=1}^{m} w_l \W_{\gamma(l)}(p_l, q_l),
\end{align}
Dualizing the linear constraint $\sqrt{W} \q = 0 $, we obtain the dual problem
\begin{align}\label{eq:DualPr}
\min_{\u \in \R^{mn}} ~\max_{\q \in\R^{nm} } ~ \left\lbrace \sum\limits_{l=1}^{m} \langle u_l, [\sqrt{W}\q]_l\rangle - \sum\limits_{l=1}^{m} w_l \W_{\gamma(l)}(p_l, q_l)\right\rbrace &= \min_{\u \in \R^{mn}} \sum_{l=1}^{m}w_l
 \W^*_{\gamma(l), p_l}([\sqrt{W}\u]_l/w_l)
\end{align}
where $\W^*_{\gamma(l), p_l}(\cdot)$ is the Fenchel--Legendre transform of $\W_{\gamma(l)}(p_l, \cdot)$,  $[\sqrt{W} \q]_i$ and $[\sqrt{W} \u]_i$ represent the $i$-th $n$-dimensional block of vectors $\sqrt{W}\q$ and $\sqrt{W}\u$ respectively. 
We apply distributed primal-dual accelerated gradient descent Algorithm~\ref{alg:main} to solve the constructed pair of primal and dual problems.

Before we move to the theoretical analysis of the algorithm, let us discuss the scalability of Algorithm~\ref{alg:main}. 
Assume that we have an arbitrary network of agents given by connected undirected graph $G=(V,E)$ without self-loops with the set $V$ of $n$ vertices and the set of edges $E = \{(i,j): i,j \in V\}$.
Then matrix $\bar W$ can be chosen as the Laplacian matrix for this graph, which is such that a) $[\bar W]_{ij} = -1$ if $(i,j) \in E$, b) $[\bar W]_{ij} = \text{deg}(i)$ if $i=j$, c) $[\bar W]_{ij} = 0$ otherwise.
Here $\text{deg}(i)$ is the degree of the node $i$, i.e., the number of neighbors of the node.
We assume that an agent $i$ can communicate with an agent $j$ if and only if the edge $(i,j)\in E$. 
In particular, the Laplacian matrix for star graph, which corresponds to the centralized distributed computations discussed in Section \ref{S:IBP_compl} is 
\begin{align}\label{w_star}
  \bar W: \{\forall i=1,\dots, m-1 ~\bar W_{ii}=1,\; \bar W_{im}=\bar W_{mi}=-1,~  \bar W_{mm}=m-1\}.
\end{align}
Algorithm \ref{alg:main} allows to perform calculations in an arbitrary connected undirected network of agents. This is in contrast to the IBP algorithm as discussed in Section \ref{S:IBP_compl}. 

\begin{algorithm}[ht]
    \caption{Accelerated Distributed Computation of Wasserstein barycenter}
    \label{alg:main}
    {
    \begin{algorithmic}[1]
    \REQUIRE Each agent $l\in V$ is assigned its measure $p_l$ and an upper bound $L$ for the Lipschitz constant of the gradient of the dual objective.
        \STATE Each agent finds $\tilde{p}_l \in S_{n}(1)$ s.t. $\|\tilde{p}_l - p_l\|_1 \le \e / 4$  and $\min_i [\tilde{p}_l]_i \ge \e / (8 n)$ and redefine $p_l:=\tilde{p}_l$. 
		E.g., $\tilde{p}_l=\left(1 - \frac{\e}{8}\right)\left(p_l + \frac{\e}{n(8 - \e)}\one \right)$.
        \STATE All agents $l \in V$ set $\gamma(l) = \frac{\e}{4mw_l\ln n}$, $\eta^{0}_l = \zeta^{0}_l = \lm^{0}_l = \hat{q}^0_l = \boldsymbol{0} \in \mathbb{R}^n$, $A_0=\alpha_0=0$ and $N$. 
        \STATE{For each agent $l \in V$:}
        \FOR{ $k=0,\dots,N-1$ }
        \STATE Find $\alpha_{k+1}$ as the largest root of the equation $A_{k+1}:=A_k+\alpha_{k+1} =2L\alpha_{k+1}^2.$
        \vspace{2mm}
        \STATE $\lm^{k+1}_l = ({\alpha_{k+1}\zeta^k_l + A_k \eta^k_l})/{A_{k+1}}.$
        \vspace{2mm}
        \STATE Calculate  $\nabla \W_{\gamma(l), p_l}^*(\lm^{k+1}_l)$:\\
         $[\nabla \W^*_{\gamma(l), p_l} (u)]_i = \sum_{j=1}^n [p_l]_j \frac{\exp(([u]_i - [C_l]_{ij})/\gamma(l))}{\sum_{r=1}^n \exp(([u]_r - [C_l]_{rj})/\gamma(l))}$\\
         \vspace{2mm}
        \STATE{Share 
        $\nabla \W_{\gamma(l), p_l}^*(\lm^{k+1}_l)$
      with $\{j \mid (i,j) \in E \}$.}
        \STATE  $\zeta^{k+1}_l = \zeta^{k}_l-\alpha_{k+1} \sum_{j=1}^{m} W_{lj} {\nabla} \W_{\gamma(l), p_j}^*(\lm^{k+1}_j).$
        \vspace{2mm}
        \STATE  $\eta^{k+1}_l = ({\alpha_{k+1}\zeta^{k+1}_l + A_k \eta^{k+1}_l })/{A_{k+1}}.$
        \vspace{2mm}
        \STATE ${q}^{k+1}_l  = \frac{1}{A_{k+1}}\sum_{l=0}^{k+1} \alpha_i q_i(\lm^{k+1}_l) = ({\alpha_{k+1} q_i([\lm_{k+1}]_l ) + A_k{q}^{k}_l})/{A_{k+1}},$ \\
        where $q_l(\cdot)$ is defined as $\nabla \W_{\gamma(l), p_l}^*(\cdot)$. 
   \ENDFOR
    \ENSURE ${\q}^{N} = [q^T_1,\dots, q^T_m]^T$.
        
    \end{algorithmic}}
\end{algorithm} 

For simplicity and comparison with the complexity of the IBP algorithm, we analyze the complexity of Algorithm~\ref{alg:main} as if it is implemented on one machine, disregarding that it can be used for distributed setup.  
\begin{theorem}\label{Th:gamma}
Algorithm \ref{alg:main} after $N=\frac{1}{\e}\sqrt{64 \chi(\bar W) m n\ln n 
\sum_{l=1}^m w_l^2 \|C_l\|^2_\infty}$ iterations
generates an $\e$-solution of problem \eqref{prob:unreg_bary}, i.e. finds a vector ${\q}^{N} = [q^T_1,\dots, q^T_m]^T$ s.t.
\begin{equation}\label{eq:WBerr}
    \sum_{l=1}^{m} w_l\W(p_l, q^{N}_l)-\sum_{l=1}^{m} w_l\W(p_l, q^*) \leq \e, \quad \|\sqrt{W}\q^{N}\|_2 \leq \e/2R,
\end{equation}
where $q^*$ is an unregularized barycenter, i.e. is a solution to \eqref{prob:unreg_bary}, and $R$ is a bound on the solution to the dual problem.
Moreover, the number of arithmetic operations is
$   O\left({ N \cdot n(mn + \nnz(\bar W))/\e }\right)$.
\end{theorem}

The proof is based on the complexity theorem of primal-dual accelerated gradient descent for a partivular pair of primal-dual problems \eqref{eq:Primal}--\eqref{eq:DualPr}.
\begin{theorem}[{see \cite[Theorem 2]{dvurechensky2017adaptive}}]
\label{Th:AGD_compl}
Let accelerated primal-dual gradient descent be applied to the pair of problems \eqref{eq:Primal}--\eqref{eq:DualPr}. Then inequalities 
\begin{align}\label{eq:sol_reg_1}
    \sum_{l=1}^m w_l \W_{\gamma(l)}(p_l, q^N_l) - \sum_{l=1}^m w_l \W_{\gamma(l)}(p_l, q^*) \leq \e/2, \quad \|\sqrt{W}\q^{N}\|_2 \leq \e/2R
\end{align}
hold no later than after $N=\sqrt{\frac{32LR^2}{\e}}$ iterations, where $L$ is the Lipschitz constant of the gradient of the dual objective and $R$ is such that $\|\u^*\|_2 \leq R$, $\u^*$ being an  optimal dual solution. 
\end{theorem}

Our next steps are to find the bounds for $L$ in the next Lemma and $R$ in Lemma \ref{Th:R}.
\begin{lemma}\label{Lm:L}
Let in \eqref{eq:Primal} $\gamma(l) = \gamma/w_l$ for some $\gamma > 0$, and $\W_{\gamma}^* (\u)$ denote the dual objective in \eqref{eq:DualPr}. Then its gradient is $L = \lambda_{\max}(W) /\gamma$-Lipschitz continuous w.r.t. $2$-norm.
\end{lemma}

\begin{proof}
Making the change of variable $[\Blm]_l = [\sqrt{W}\u]_l/w_l$, by the chain rule, the $i$-th $n$-dimensional block of $\nabla \W^*_\gamma (\u)$ is
\begin{align}\label{eq:DualObjGrad}
    \left[\nabla \W^*_{\gamma} (\u) \right]_i = 
    \left[\nabla \sum_{l=1}^{m} w_l\W^*_{\gamma(l), p_l}([\sqrt{W}\u]_l/w_l) \right]_i = \sum_{l=1}^{m}\sqrt{W}_{il}  \nabla \W_{\gamma(l), p_l}^*(\lm_l), \; i=1,\dots,m.
\end{align}
Thus,
\begin{align*}
\|\nabla \W^*_{\gamma}(\u_1) - \nabla \W^*_{\gamma }(\u_2)\|^2_2
&\stackrel{\eqref{eq:DualObjGrad} }{=}\left\|
\sqrt{W} \left(
\begin{aligned}
\nabla \W^*_{\gamma(1), p_1}&([\Blm_1]_1)\\
&...\\
\nabla \W^*_{\gamma(m), p_m}&([\Blm_1]_m)
\end{aligned}
\right) - 
\sqrt{W} \left(
\begin{aligned}
\nabla \W^*_{\gamma(1), p_1}&([\Blm_2]_1)\\
&...\\
\nabla \W^*_{\gamma(m), p_m}&([\Blm_2]_m)
\end{aligned}
\right)
\right\|_2^2 \\
&\leq (\lambda_{\max}(\sqrt{W}))^2 \left\|
\begin{aligned}
\nabla \W^*_{\gamma(1), p_1}([\Blm_1]_1) &- \nabla \W^*_{\gamma(m),p_1}([\Blm_2]_1)\\
&...\\
\nabla \W^*_{\gamma(1), p_m}([\Blm_1]_m) &- \nabla \W^*_{\gamma(m),p_m}([\Blm_2]_m)
\end{aligned}
\right\|_2^2 \\
&=(\lambda_{\max}(\sqrt{W}))^2 \sum_{i=1}^m \left\| \nabla \W^*_{\gamma(i),p_i}([\Blm_1]_i) - \nabla \W^*_{\gamma(i),p_i}([\Blm_2]_l) \right\|_2^2\\
& \leq (\lambda_{\max}(\sqrt{W}))^2 \sum_{i=1}^m \frac{1}{\gamma^2(i)} \left\| [\Blm_1]_l - [\Blm_2]_l \right\|_2^2 \\
& = {(\lambda_{\max}(\sqrt{W}))^2} \sum_{i=1}^m \frac{1}{\gamma^2(i)} \left\| [\sqrt{W} (\u_1 - \u_2)]_i/w_i \right\|_2^2,
\end{align*}
where we used notation $[\Blm]_i = [\sqrt{W} \u]_i/w_i$, the definition of matrix $\sqrt{W}$, $1/\gamma(i)$-Lipschitz continuity of $\nabla \W_{\gamma(i),p_i}^*(\cdot)$ for all $i=1,\dots, m$ by \cite[Theorem 2.4]{cuturi2016smoothed}.
Since $\gamma(i) = \gamma/w_i$, $i=1,\dots, m$, we obtain 
\begin{align}
\|\nabla \W^*_{\gamma}(\u_1) - \nabla \W^*_{\gamma }(\u_2)\|^2_2    
& \leq \frac{(\lambda_{\max}(\sqrt{W}))^2}{ \gamma^2} \left\| \sqrt{W} (\u_1 - \u_2) \right\|_2^2\\
& \leq \frac{(\lambda_{\max}(\sqrt{W}))^4}{\gamma^2} \left\|\u_1 - \u_2\right\|_2^2,
\end{align}
 Since $(\lambda_{\max}(\sqrt{W}))^4 = (\lambda_{\max}(W))^2$, we get the statement of Lemma. 
\end{proof}

The following Lemma is inspired by \cite{lan2017communication}.
\begin{lemma}\label{Th:R}
Let $\q^*_\gamma$ be the optimal solution of problem \eqref{prob:reg_bary} with minimal 2-norm, then there exists an optimal dual solution $\Blm^* = [\lm^*_1, \dots, \lm^*_m]$ for problem \eqref{eq:DualPr} satisfying $\| \Blm^* \|_2 \leq R$ with
\begin{align}
    R^2  = \frac{2n\sum_{l=1}^m w_l^2 \|C_l\|^2_\infty}{{\lm}^+_{min}(W)}.
\end{align}
Here ${\lm}^+_{min}(W)$ is the minimal positive eigenvalue of the matrix $W$.
\end{lemma}
\begin{proof}
Recall that $\W_{\gamma}(\p, \cdot)$ denotes the objective value in the primal problem \eqref{eq:Primal}. Then
\begin{align}
    - \W_{\gamma}(\p, \q^*_\gamma) &= \la \Blm^*, \sqrt{W} \q_\gamma^*\ra - \W_{\gamma}(\p, \q^*_\gamma) 
    = \W_{\gamma, \p}^*( \sqrt{W} \Blm^*) = \max_{\substack{q_l \in S_n(1), \\ l=1,\dots,m}} \{\la \Blm^*_l, \sqrt{W} \q_l\ra - \W_{\gamma}(\p, \q)\} \notag \\
    &\geq \la \Blm^*, [\sqrt{W} \q]_l\ra - \W_{\gamma}(\p, \q)
    = \la \Blm^*, \sqrt{W} \q - \sqrt{W} \q^*_\gamma \ra - \W_{\gamma}(\p, \q) \notag \\
    &= 
    \la \sqrt{W}\Blm^*, \q^*_\gamma - \q\ra - \W_{\gamma, \p}(\q), \notag
\end{align}
where we used $W^T=W$ and $\q^*_\gamma$ is the regularized barycenter.\\
From this inequality and using the convexity of $\W_{\gamma}(\p, \q)$ we have $\nabla \W_{\gamma}(\p, \q^*) = -\sqrt{W}\Blm^*$.\\
Then we have
\begin{multline}
  \norm{\nabla \W_{\gamma}(\p, \q^*)}_2^2 
    =  \norm{-\sqrt{W}\Blm^*}_2^2 
    = \la \sqrt{W} \Blm^*, \sqrt{W} \Blm^* \ra 
    = \la \Blm^*, W \Blm^* \ra 
    \geq \lm^+_{min}(W) \norm{\Blm^*}_2^2,
\end{multline}
where the last  inequality holds due to $\Blm^* \in \rm \left(Ker(\sqrt{W})\right)^\perp$

Hence, 
we get
\begin{align}
  \| \Blm^* \|_2^2 \leq  R^2 =  \frac{\norm{\nabla \W_{\gamma}(\p,\q_\gamma^*)}_2^2}{{\lm}^+_{min}(W)} =  \frac{\sum_{l=1}^m w_l^2\norm{\nabla \W_{\gamma(l) }(p_l,q^*_\gamma)}_2^2}{{\lm}^+_{min}(W)}
\end{align}

Let us now estimate $\norm{\nabla \W_{\gamma(l) }(p_l,q^*_\gamma)}_2^2$. From \eqref{eq:reg_OT}, we can construct the dual problem to the regularized optimal transport problem
\begin{align}
    \W_{\gamma(l)}(p,q) = \min_{\pi \in  \Pi(p,q)} \left\{ \la \pi, C \ra + \gamma(l) H(\pi) \right\} &= \max_{ \mu,\nu} \left\{ -\la \mu, p \ra - \la \nu,q\ra - \frac{\gamma(l)}{e} e^{-\frac{\mu}{\gamma(l)}} e^{-\frac{C}{\gamma(l)}} e^{-\frac{\nu}{\gamma(l)}} \right\} \notag 
\end{align}
By \cite[Proposition 2.3]{cuturi2016smoothed}, $\nabla_q\W_{\gamma(l)}(p,q) = -\nu^*$, where $\nu^*$ is the solution to the dual problem satisfying $\la \nu^*, \one \ra = 0$. Hence, $\min_{i=1,...,n}\nu^*_i \leq 0 \leq \max_{i=1,...,n}\nu^*_i$. As it follows from \cite[Lemma 1]{dvurechensky2018computational}
\[
\max_{i=1,...,n}\nu^*_i - \min_{i=1,...,n}\nu^*_i \leq \|C_l\|_\infty - \gamma(l) \ln\left({\rm min}_{i} [p_l]_i\right) =\|C_l\|_\infty + \gamma(l)  \ln\left(\frac{8n}{\e}\right),
\]
where we used the fact that the $p_l$ was redefined in Algorithm~\ref{alg:main} in such a way that ${\rm min}_{i} [p_l]_i \geq \frac{\e}{8n}$ and also that our variable $\nu^*$ and their dual variable $u^*$ satisfy $u^* = - \frac{\nu^*}{\gamma(l)}- \frac{1}{2}$.
Making the same arguments as in the proof \cite[Lemma 3.2.]{lin2019efficient}, we obtain from the above two facts that 
\[
\|\nu^*\|_2\leq \|\nu^*\|_\infty \sqrt{n} \leq \|C_l\|_\infty\sqrt{n} + \frac{\gamma \sqrt{n}}{w_l}\ln\left(\frac{8n}{\e}\right),
\]
where we used that $\gamma(l)=\gamma/w_l$.
Thus, 
\[
\sum_{l=1}^m w_l^2\norm{\nabla \W_{\gamma(l) }(p_l,q^*_\gamma)}_2^2 \leq 2n\sum_{l=1}^m w_l^2 \|C_l\|^2_\infty  + 2n\gamma^2 \ln^2\left(\frac{8n}{\e}\right).
\]
Since $\gamma$ is chosen proportional to $\e$ which is small, we can neglect the second term in comparison with the first one.

\end{proof}

\begin{proof}[Proof of Theorem 3]
By Theorem \ref{Th:AGD_compl}, we have
\begin{align}\label{eq:sol_reg}
    \sum_{l=1}^m w_l \W_{\gamma(l)}(p_l, q^N_l) - \sum_{l=1}^m w_l\W_{\gamma(l)}(p_l, q^*) \leq \e/2, \quad \|\sqrt{W}\q^{N}\|_2 \leq \e/2R
\end{align}
Since $KL(\pi|\theta)\geq 0$,  we have
\begin{align}\label{eq:from_unreg_to_reg}
    \sum_{l=1}^mw_l \W(p_l, q^N_l) - \sum_{l=1}^m w_l\W(p_l, q^*) \leq \sum_{l=1}^m w_l\W_{\gamma(l)}(p_l, q^N_l)  - \sum_{l=1}^m w_l\W(p_l,q^*).
\end{align}

By definition of the objective $\W_{\gamma}(\p, \cdot)$ in \eqref{eq:Primal}, 
\begin{align}\label{eq:wass_gamma}
    \sum_{l=1}^m w_l \W_{ \gamma(l)}(p_l, q^*) &= \min_{\substack{q_1=\dots =q_m \\ q_1,\dots q_m \in S_n(1)}}  ~ \sum_{l=1}^m  ~ w_l\min_{\pi \in \Pi(p_l, q_l)}\left\{\sum_{i,j=1}^n  C_{ij}\pi_{ij}  + \gamma KL(\pi|\theta)\right\}\notag \\  
    &\leq \min_{\substack{q_1=\dots =q_m \\ q_1,\dots q_m \in  S_1(n)}}  ~ \sum_{l=1}^m  ~ \left\{\min_{\pi \in \Pi(p_l, q_l)} w_l\sum_{i,j=1}^n  C_{ij}\pi_{ij} + w_l\max_{\pi \in \Pi(p_l, q_l)}\gamma KL(\pi|\theta)\right\}  \notag \\
    &\leq \min_{\substack{q_1=\dots =q_m \\ q_1,\dots q_m \in  S_n(1)}}  ~ \sum_{l=1}^m  ~ \left\{\min_{\pi \in \Pi(p_l, q_l)} \sum_{i,j=1}^n  C_{ij}\pi_{ij} + 2\sum_{l=1}^m w_l\gamma(l) \ln n \right\} \notag\\
    &\leq \min_{\substack{q_1=\dots =q_m \\ q_1,\dots q_m \in  S_n(1)}}  ~ \sum_{l=1}^m \W(p_l, q_l) ~  + 2 \ln n \sum_{l=1}^m w_l \gamma(l)  \notag \\
    &= \sum_{l=1}^m \W(p_l, q^*) + 2\ln n\sum_{l=1}^m w_l\gamma(l) ,
\end{align}
where we chose  $\theta_{ij} = 1/n^2$ for all $i, j=1, \dots n$ so $KL(\pi|\theta) \in [0, 2\ln n]$.

Substituting \eqref{eq:wass_gamma} in \eqref{eq:from_unreg_to_reg} we get

\begin{align}
    \sum_{l=1}^m w_l\W_{\gamma(l)}(p_l, q^N_l) - \sum_{l=1}^m w_l \W(p_l, q^*) \leq &\sum_{l=1}^m w_l\W_{\gamma(l)}(p_l, q^N_l)  \notag \\
    & - \sum_{l=1}^m w_l\W_{\gamma(l)}(p_l, q^*)  + 2\ln n\sum_{l=1}^m w_l \gamma(l)
\end{align}
Using this inequality and \eqref{eq:sol_reg} we get
\begin{align}
    \sum_{l=1}^m w_l  \W(p_l, q_l^N) - \sum_{l=1}^m w_l \W(p_l, q^*) \leq \e/2 +2\ln n\sum_{l=1}^m w_l\gamma_l. 
\end{align}
Since $\gamma(l) = \gamma/w_l$ with $\gamma = \e/(4m \ln n)$, we obtain thath the inequality \eqref{eq:sol_reg} hold. Combining the values of $\gamma$, $L$ from Lemma \ref{Lm:L}, $R$ from Lemma \ref{Th:R} with the estimate for $N$ in Theorem \ref{Th:AGD_compl} and the fact that $\chi(W) = \chi(\bar W)$, we obtain an estimate for the number of iterations of the algorithm.
Let us estimate the complexity of the algorithm. For each $l$ we need to calculate the gradient $\W^*_{\gamma(l), p_l}(\cdot)$, which  requires $O(n^2)$ arithmetic operations. 
To calculate $\sum_{j=1}^{m} W_{lj} {\nabla} \W_{ \gamma(l), p_l)}^*(\lm^{k+1}_j)$ one needs $O(n\cdot \nnz(\bar W_l))$ arithmetic operations, where $\nnz(\bar{W}_l)$ is the number of non-zero elements in matrix $\bar W$ in the $l$ row. More precisely, the dimension of $\nabla \W^*_{\gamma(l), p_l}(\cdot)$ is $n$ and the matrix $W_{lj}$ is diagonal  for each $l,j=1,\dots, m$. Using definition of $W$ we get that the complexity of calculating the gradient. Other operations require $O(n)$ operations. Hence, the complexity of one iteration is 
\begin{align}
    O\left(mn^2+\sum_{l=1}^m n\cdot \nnz(\bar W_l)\right) = O\left(mn^2 + n \cdot \nnz(\bar W)\right)    
\end{align}
and the total complexity follows from multiplying this value by $N$.

\end{proof}

Let us make a couple of remarks.
As for the choice of $\bar{W}$ one can show (by using graph sparsificators) that it can be chosen such that  $\chi(W)=\chi(\bar{W}) = O(\mathrm{Poly}(\ln(m)))$ and $\nnz(\bar{W}) = O(m \mathrm{Poly}(\ln(m)))$. 
For details on the graph sparsificators we refer to \cite{vaidya1990solving,bern2006support,spielman2014nearly}.
In the simple case of equal weights $w_l = \frac{1}{m}$,  
 the complexity of approximating non-reguarized barycenter by accelerated gradient descent can be estimated as $\widetilde{O}(mn^{2.5}/\e)$. In the distributed setting, each of $m$ nodes makes $\widetilde{O}(n^{2.5}/\e)$ arithmetic operations, while the number of communications rounds is $\widetilde{O}(\sqrt{n}/\e)$. The case of general weights is interesting with respect to bootstrap procedure allowing to construct confidence sets for barycenter \cite{ebert2017construction}.

There is an interesting connection of our work with \cite{cuturi2016smoothed}.
Consider a particular case of equal weights $w_1=\dots=w_m =1$ and matrix $\bar{W}$ corresponding to the star graph topology. Then the dual problem \eqref{eq:Primal} has the form
\begin{align}\label{eq:dual_2}
    \min_{\lm_1, \dots, \lm_m \in \R^n} \left\{ \sum_{l=1}^{m-1}\W^{*}_{\gamma, p_l} \left(\lm_l - \frac{1}{m}\lm_m\right) + \W^{*}_{\gamma, p_m} \left(\frac{m-1}{m}\lm_m - \sum_{l=1}^{m-1}\lm_l\right) \right\}
\end{align}
Changing the variable $ \hat{\lm}_l = \lm_l -\frac{1}{m}\lm_m$ we come to the following formulation.
\begin{align}\label{eq:dual_1}
   \min_{\hat \lm_1, \dots, \hat \lm_{m-1} \in \R^n} \sum_{l=1}^{m-1}\W^*_{\gamma, p_l} (\hat \lm_l) +\W^*_{\gamma, p_m}\left(-\sum_{l=1}^{m-1}\hat \lm_l\right)
\end{align}
Hence, the approach presented in \cite{cuturi2016smoothed}, in Theorem 3.1 is the particular case for the approach described above, corresponding to the the star graph.

\section*{Conclusion}

\ag{In this paper we show that IBP algorithm from \cite{benamou2015iterative} for Wasserstein barycenter problem can be implemented in a centralized distributed manner such that each node fulfils $\widetilde{O}\left(n^2/\e^2\right)$ arithmetic operations and the number of communication rounds is $\widetilde{O}\left(1/\e^2\right)$. We note that proper proximal envelope of this algorithm can sometimes accelerate this bounds in terms of the dependence of $\e$. We also describe accelerated primal-dual gradient  algorithm for the same problem. The proposed algorithm can be implemented in a more general decentralized distributed setting such that, to find an $\e$-approximation for the non-regularized barycenter, each node performs $\widetilde{O}(n^{2.5}/\e)$ arithmetic operations and the number of communication rounds is $\widetilde{O}(\sqrt{n}/\e)$.}

\textbf{Acknowledgments.} 
This research was funded by Russian Science Foundation (project 18-71-10108).
The work of Alexey Kroshnin was also conducted within the framework of the HSE University Basic Research Program and funded by the Russian Academic Excellence Project '5-100'.

\bibliographystyle{plainnat}
\bibliography{wass,PD_references,time_varying}


\end{document}

%% file: IBP_short.tex
In this section we provide theoretical analysis of the Iterative Bregman Projections algorithm \cite{benamou2015iterative} for regularized Wasserstein barycenter and obtain iteration complexity $O\left(\frac{c}{\gamma \e}\right)$ with $c := \max_{l=1,...,m}\|C_l\|_{\infty}$. Then we estimate the bias itroduced by regularization and estimate  the value of $\gamma$ to obtain an $\e$-approximation for the non-regularized barycenter. Combining this result with the iteration complexity of IBP, we obtain complexity 
$\widetilde{O}\left(\frac{c^2mn^2}{\e^2}\right)$
for approximating   non-regularized barycenter by the IBP algorithm. \gav{This algorithm can be implemented in a centralized distributed manner such that each node performs $\widetilde{O}\left(\frac{c^2n^2}{\e^2}\right)$ arithmetic operations and the number of communication rounds is $\widetilde{O}\left(\frac{c^2}{\e^2}\right)$}.

\subsection{Convergence of IBP for regularized barycenter}\label{ConvergenceIBP}
In this subsection we analyze Iterative Bregman Projection Algorithm \cite[Section 3.2]{benamou2015iterative} and analyze its complexity for solving problem~\eqref{prob:reg_bary}. We slightly reformulate this problem as
\begin{align}
    \min_{q \in S_n(1)} \sum_{l=1}^m w_l \W_\gamma (p_l,q) &= \min_{\substack{q \in S_n(1), \\ \pi_l \in \Pi(p_l, q),\\
    l = 1, \dots, m }} \sum_{l = 1}^m w_l \bigl\{\la \pi_l, C_l \ra + \gamma H(\pi_l) \bigr\} \notag\\ 
    &= \min_{\substack{ \pi_l \in \R_+^{n \times n}\\ \pi_l \one = p_l,\; \pi_l^T \one = \pi_{l+1}^T \one, \\
    l=1,\dots,m}} \sum_{l = 1}^m w_l \bigl\{\la \pi_l, C_l \ra + \gamma H(\pi_l) \bigr\}\label{prob:reg_bary_2}
\end{align}
and construct its dual (see the details below). 
To solve the dual problem we reformulate the IBP algorithm as Algorithm~\ref{Alg:dual_IBP}\footnote{\gav{In the original paper \cite{benamou2015iterative} there were misprints in description of IBP. Correct author's variant can be found in \url{https://github.com/gpeyre/2014-SISC-BregmanOT}. In Algorithm~\ref{Alg:dual_IBP} we use different denotations. First of all, our $u$ corresponds to $\ln v$ from \cite{benamou2015iterative} and our $v$ corresponds to $\ln u$ from \cite{benamou2015iterative}. Secondly, our transport plan matrix equals to transpose transport plan matrix from \cite{benamou2015iterative}. Thirdly, we build a little bit different dual problem, by introducing additional constraint $\sum_{l=1}^m w_lv_l = 0$, see Lemma~\ref{lemma:dual_2}. This allows us to simplify calculations in line 3 of Algorithm~\ref{Alg:dual_IBP}.}}. Notably, our reformulation of the IBP algorithm allows to solve simultaneously the primal and dual problem and has an adaptive stopping criterion (see line 7), which does not require to calculate any objective values.   

\begin{algorithm}[H]
    \caption{Dual Iterative Bregman Projection }
    \label{Alg:dual_IBP}
    \begin{algorithmic}[1]
        \REQUIRE $C_1, \dots, C_m$, $p_1, \dots, p_m$, $\gamma > 0$, $\e' > 0$
        \STATE $u_l^0 := 0$, $v_l^0 := 0$, $K_l := \exp\left(-\tfrac{C_l}{\gamma}\right)$, $l = 1, \dots, m$ 
        \REPEAT \STATE 
                $v_l^{t + 1} := \sum_{k = 1}^m w_k \ln K_k^T e^{u_k^t} - \ln K_l^T e^{u_l^t}, \quad
                \u^{t + 1} := \u^t$ 
                \STATE
                $t := t + 1$
                \STATE  $u^{t + 1}_l := \ln p_l - \ln K_l e^{v_l^t}, \quad 
                \v^{t + 1} := \v^t$ 
                \STATE $t := t + 1$
        \UNTIL{$\sum_{l = 1}^m w_l \norm{B_l^T(u_l^t, v_l^t) \one - \bar{q}^t}_1 \le \e'$, where $B_l(u_l, v_l)=\diag\left(e^{u_l}\right) K_l \diag\left(e^{v_l}\right)$, $\bar{q}^t := \sum_{l = 1}^m w_l B_l^T(u_l^t, v_l^t) \one$}
        \ENSURE $B_1(u^t_1, v^t_1), \dots, B_m(u^t_m, v^t_m)$  
    \end{algorithmic}
\end{algorithm}

Before we move to the analysis  of the algorithm let us discuss the scalability of this algorithm by using centralized distributed computations framework. This framework includes a master and slave nodes. Each $l$-th slave node stores data $p_l$, $C_l$, $K_l$ and variables $u_l^t$, $v_l^t$. On each iteration $t$, it  calculates $K_l^Te^{u_l^t}$ and sends it to the master node, which aggregates these products to the sum $\sum_{k = 1}^m w_k \ln K_k^T e^{u_k^t}$ and sends this sum back to the slave nodes. Based on this information, slave nodes update $v_l^t$ and $u_l^t$. So, the main computational cost of multiplying a matrix by a vector, can be distributed on $m$ slave nodes and the total working time will be smaller. It is not clear, how this algorithm can be implemented on a general network, for example when the data is produced by a distributed network of sensors without one master node. In contrast, as we illustrate in Section \ref{S:APDAGD_compl}, the alternative accelerated-gradient-based approach can be implemented on an arbitrary network.   


\begin{theorem}
\label{thm:complexity_IBP}
    For given $\e'$ Algorithm~\ref{Alg:dual_IBP} stops in number of iterations $N$ satisfying 
    \begin{align*}
        N \le 4 + \frac{44 R_v}{\e'}
        = O\left(\frac{\max_l \norm{C_l}_\infty}{\gamma \e'}\right).
    \end{align*}
    It returns $B_1, \dots, B_m$ s.t.
    \[
    \sum_{l = 1}^m w_l \norm{B_l \one - \bar{q}}_1 \le \e', \quad \bar{q} = \sum_{l = 1}^m w_l B_l \one,
    \]
    and
    \begin{equation}
        \sum_{l = 1}^m w_l \left(\left\la C_l, B_l \right\ra + \gamma H(B_l)\right) 
        - \sum_{l = 1}^m w_l \left(\left\la C_l, \pi_{\gamma,l}^* \right\ra + \gamma H(\pi^*_{\gamma, l})\right)
        \le \max_l \norm{C_l}_\infty \e',
        \label{regularized_residual}
    \end{equation}
    where $\Bpi_\gamma^* = [\pi^*_{\gamma, 1}, \dots, \pi^*_{\gamma, m}]$ is a solution of problem~\eqref{prob:reg_bary_2}. 
\end{theorem}

Our first step is to recall the IBP algorithm from \cite{benamou2015iterative}, formulate the dual problem for \eqref{prob:reg_bary_2} and show that our Algorithm~\ref{Alg:dual_IBP} solves this dual problem and is equivalent to the IBP algorithm. 

Following the approach from \cite{benamou2015iterative} we present the problem~\eqref{prob:reg_bary_2} in a Kullback--Leibler projection form.
\begin{align}\label{prob:bary_KL}
    \min\limits_{\Bpi \in \C_1 \cap  \C_2} \sum_{l = 1}^m w_l KL\left(\pi_l | \theta_l\right),
\end{align}
for $\theta_l = \exp\left(-\frac{C_l}{\gamma}\right)$ and the following affine convex sets $\C_1$ and $\C_2$  
\begin{align}
    \C_1 &= \left\{\Bpi = [\pi_1, \dots, \pi_m] : \forall l ~ \pi_l \one = p_l \right\},  \notag \\
    \C_2 &= \left\{\Bpi = [\pi_1, \dots, \pi_m] : \exists q \in S_n(1)\; \forall l ~ \pi_l^T \one = q \right\}. \label{eq:C1C2Def}
\end{align}
Algorithm~\ref{Alg:IBP} introduced in \cite{benamou2015iterative} consists in alternating projections to sets $\C_1$ and $\C_2$ w.r.t. Kullback--Leibler divergence, and is a generalization of Sinkhorn's algorithm and a particular case of Dykstra's projection algorithm. 

\begin{algorithm}[H]
    \caption{Iterative Bregman Projections, see \cite{benamou2015iterative}}
    \label{Alg:IBP}
    \begin{algorithmic}[1]
        \REQUIRE Cost matrices $C_1, \dots, C_m$, probability measures $p_1, \dots, p_m$, $\gamma > 0$, starting transport plans $\{\pi_l^0\}_{l=1}^m: \pi_l^0 := \exp\left(-\frac{C_l}{\gamma}\right)$, $l = 1, \dots, m$ 
        \REPEAT
            \IF{$t \bmod 2 = 0$}
                \STATE $\Bpi^{t + 1} := \argmin\limits_{\Bpi \in \C_1} \sum_{l = 1}^m w_l KL\left(\pi_l | \pi_l^t\right)$
            \ELSE
                \STATE $\Bpi^{t + 1} := \argmin\limits_{\Bpi \in \C_2} \sum_{l = 1}^m w_l KL\left(\pi_l | \pi_l^t\right)$
            \ENDIF
            \STATE $t := t + 1$
        \UNTIL{Converge}
        \ENSURE $\Bpi^t$    
    \end{algorithmic}
\end{algorithm}

As we will show below, this algorithm is equivalent to alternating minimization in dual problem of~\eqref{prob:reg_bary_2} presented in the next lemma.

\begin{lemma}
\label{lemma:dual_2}
    The dual problem of~\eqref{prob:reg_bary_2} is (up to a multiplicative constant)
    \begin{equation}
    \label{prob:dual_2}
        \min_{\substack{\u, \v\\ \sum_l w_l v_l = 0}} f(\u, \v) \quad \text{where} \quad f(\u, \v) := \sum_{l = 1}^m w_l \bigl\{\la \one, B_l(u_l, v_l) \one \ra - \la u_l, p_l \ra\bigr\},
    \end{equation}
    $\u = [u_1, \dots, u_m]$, $\v = [v_1, \dots, v_m]$, $u_l, v_l \in \R^n$, and
    \begin{equation}\label{eq:B_def}
        B_l(u_l, v_l) := \diag\left(e^{u_l}\right) K_l \diag\left(e^{v_l}\right), \quad
        K_l := \exp\left(-\frac{C_l}{\gamma}\right).
    \end{equation}
    Moreover, solution $\Bpi_\gamma^*$ to~\eqref{prob:reg_bary_2} is given by the formula
    \begin{equation}
    \label{eq:dual_to_primal}
        [\Bpi_\gamma^*]_l = B_l(u^*_l, v^*_l),
    \end{equation}
    where $(\u^*, \v^*)$ is a solution to the problem~\eqref{prob:dual_2}.
\end{lemma}

\begin{proof}
    The Lagrangian for~\eqref{prob:reg_bary_2} is equal to
    \begin{align*}
        L(\Bpi; \Blm, \Bmu) 
        & = \sum_{l = 1}^m w_l \bigl\{\la \pi_l, C_l \ra + \gamma H(\pi_l)\bigr\} + \sum_{l = 1}^m \la \lambda_l, \pi_l \one - p_l \ra + \sum_{l = 1}^m \la \mu_l, \pi_{l + 1}^T \one - \pi_l^T \one \ra \\
        & = \sum_{l = 1}^m \left[w_l \bigl\{\la \pi_l, C_l \ra + \gamma \la \pi_l, \ln \pi_l - \one \one^T \ra \bigr\} + \la \lambda_l, \pi_l \one - p_l \ra + \la \mu_{l - 1} - \mu_l, \pi_l^T \one \ra \right],
    \end{align*}
    where $\Blm = [\lambda_1, \dots, \lambda_m]$, $\Bmu = [\mu_1, \dots, \mu_m]$, $\lambda_l, \mu_l \in \R^n$ with convention $\mu_0 \equiv \mu_m \equiv 0$.
    Using the change of variables $u_l := - {\lambda_l}/{(w_l \gamma)}$ and $v_l := {(\mu_l - \mu_{l - 1})}/{(w_l \gamma)}$
    we obtain
    \begin{equation}
    \label{eq:lagrangian}
        L(\Bpi; \u, \v)
        = \gamma \sum_{l = 1}^m w_l \left\{\left\la \pi_l, \frac{C_l}{\gamma} + \ln \pi_l - \one \one^T - u_l \one^T - \one v_l^T \right\ra + \la u_l, p_l \ra \right\}.
    \end{equation}
    Notice that $\sum_{l = 1}^m w_l v_l = 0$ and this condition allows uniquely reconstruct $\mu_1, \dots, \mu_m$.
    Then by min-max theorem
    \begin{align*}
        \min_{\substack{\Bpi : \pi_l \in \R_+^{n \times n}\\ \pi_l \one = p_l,\; \pi_l^T \one = \pi_{l+1}^T \one}} 
        & \sum_{l = 1}^m w_l \bigl\{\la \pi_l, C_l \ra + \gamma H(\pi_l) \bigr\} \\
        & = \min_{\Bpi : \pi_l \in \R_+^{n \times n}} \max_{\substack{\u, \v\\ \sum_l w_l v_l = 0}} L(\Bpi; \u, \v) 
        = \max_{\substack{\u, \v\\ \sum_l w_l v_l = 0}} \min_{\Bpi : \pi_l \in \R_+^{n \times n}} L(\Bpi; \u, \v) \\
        & = \max_{\substack{\u, \v\\ \sum_l w_l v_l = 0}} \gamma \sum_{l = 1}^m 
        w_l \left\{\min_{\pi_l \in \R_+^{n \times n}} \left\la \pi_l, \frac{C_l}{\gamma} + \ln \pi_l - \one \one^T - u_l \one^T - \one v_l^T \right\ra + \la u_l, p_l \ra\right\}.
    \end{align*}
    By straightforward computations and the definition \eqref{eq:B_def} of $B_l(u_l, v_l)$ we obtain
    \begin{multline*}
        \min_{\pi_l \in \R_+^{n \times n}} \left\la \pi_l, \frac{C_l}{\gamma} + \ln \pi_l - \one \one^T - u_l \one^T - \one v_l^T\right\ra \\
        = - \left\la \one, \exp\left(u_l \one^T - \frac{C_l}{\gamma} + \one v_l^T \right) \one \right\ra
        = - \left\la \one, B_l(u_l, v_l) \one \right\ra,
    \end{multline*}
    and the minimum is attained at point $\pi_l = B_l(u_l, v_l)$.
    Thus we have
    \begin{align*}
        \min_{\substack{\Bpi : \pi_l \in \R_+^{n \times n}\\ \pi_l \one = p_l,\; \pi_l^T \one = \pi_{l+1}^T \one}} \sum_{l = 1}^m w_l \bigl\{\la \pi_l, C_l \ra + \gamma H(\pi_l) \bigr\}
        & = \max_{\substack{\u, \v\\ \sum_l w_l v_l = 0}} \gamma \sum_{l = 1}^m w_l \left\{- \la \one, B_l(u_l, v_l) \one \ra + \la u_l, p_l \ra + 1\right\} \\
        & = - \gamma \min_{\substack{\u, \v\\ \sum_l w_l v_l = 0}} \sum_{l = 1}^m w_l \left\{\la \one, B_l(u_l, v_l) \one \ra - \la u_l, p_l \ra\right\}.
    \end{align*}
    Consequently, the dual problem to~\eqref{prob:reg_bary_2} is equivalent to~\eqref{prob:dual_2}, and solution to~\eqref{prob:reg_bary_2} has the form $[\Bpi^*_\gamma]_l = B_l(u^*_l, v^*_l)$.
\end{proof}

The following lemma shows that Algorithms~\ref{Alg:IBP} is equivalent to alternating minimization in dual problem~\eqref{prob:dual_2} (what is a general fact for Dykstra's algorithm).

\begin{lemma}
\label{thm:IBP_to_dual}
    Sequence $\{\Bpi^t\}_{t \ge 0}$ generated by Algorithm~\ref{Alg:IBP} has a form $\pi_l^t = B_l(u^t_l, v^t_l)$, where $B_l(\cdot, \cdot)$ is defined in \eqref{eq:B_def} and
    \begin{align}
        &\u^0 = \v^0 = 0, &&&\notag \\
    	&\u^{t+1} := \argmin_\u f(\u, \v^t), & \v^{t+1} := \v^t & & t \bmod 2 = 0, \label{eq:u_next} \\
    	&\v^{t+1} := \argmin_{\v : \sum_{l = 1}^m w_l v_l = 0} f(\u^t, \v), & \u^{t+1} := \u^t & & t \bmod 2 = 1. \label{eq:v_next}
    \end{align}
\end{lemma}

\begin{proof}
    Let us prove it by induction. For $t = 0$ it is obviously true. Assume it holds for some $t \ge 0$. Then
    \[
    KL\left(\pi_l | \pi_l^t\right)
    = H(\pi_l) - \la \pi_l, \ln \pi_l^t \ra + \la \pi_l^t, \one \one^T \ra
    = H(\pi_l) + \left\la \pi_l, \frac{C_l}{\gamma} - u_l^t \one^t - \one (v_l^t)^T \right\ra + \la \pi_l^t, \one \one^T \ra.
    \]
    Therefore, 
    \begin{gather*}
        \sum_{l = 1}^m w_l KL\left(\pi_l | \pi_l^t\right) \to \min_{\Bpi \in \C} \\
        \Longleftrightarrow \sum_{l = 1}^m w_l \left( \la C_l, \pi_l \ra - \gamma \la \pi_l \one - p_l, u_l^t \ra - \gamma \la \pi_l^T \one, v_l^t \ra \right) = L(\Bpi; \u^t, \v^t) \to \min_{\Bpi \in \C},
    \end{gather*}
    where $L$ comes from~\eqref{eq:lagrangian}.
    
    Thus for even $t$
    \[
    \Bpi^{t + 1} = \argmin_{\Bpi \in \C_1} L(\Bpi; \u^t, \v^t),
    \]
    Lagrangian for this problem has form $L(\Bpi, \u, \v^t)$, hence the dual problem is
    \[
    \gamma f(\u, \v^t) \to \min_{\v}.
    \]
    and as $\u^{t+1} = \argmin_{\u} f(\u, \v^t)$,
    \[
    \pi_l^{t + 1} = B_l(u^{t+1}_l, v^t_l). 
    \]
    Similarly, for odd $t$ 
    \[
    \Bpi^{t + 1} = \argmin_{\Bpi \in \C_2} L(\Bpi; \u^t, \v^t),
    \]
    with Lagrangian $L(\Bpi, \u^t, \v)$ and dual problem
    \[
    \gamma f(\u^t, \v) \to \min_{\v : \sum_{l = 1}^m w_l v_l = 0}.
    \]
    Consequently,
    \[
    \pi_l^{t + 1} = B_l(u^{t+1}_l, v^t_l). 
    \]
\end{proof}

The next lemma gives us explicit recurrent expressions for $\u^t$ and $\v^t$ defined in the previous lemma. Equation~\eqref{eq:u_cond} immediately follows from \cite[Proposition 1]{benamou2015iterative}, and equation~\eqref{eq:v_cond} is a reformulation of \cite[Proposition 2]{benamou2015iterative}.

\begin{lemma}
\label{lemma:IBP_explicit_step}
    Equation~\eqref{eq:u_next} for even $t$ is equivalent to
    \begin{equation}
    \label{eq:u_cond}
        u^{t+1}_l 
        = u^t_l + \ln p_l - \ln \left(B_l\left(\u^t, \v^t\right) \one\right) 
        = \ln p_l - \ln K_l e^{v_l^t},
    \end{equation}
    and equation~\eqref{eq:v_next} for odd $t$ is equivalent to
    \begin{equation}
    \label{eq:v_cond}
        v^{t + 1}_l
        = v^t_l + \ln q^{t + 1} - \ln q^t_l 
        = \sum_{k = 1}^m w_k \ln K_k^T e^{u_k^t} - \ln K_l^T e^{u_l^t},
    \end{equation}
    where $q^t_l := B_l^T(u^t_l, v^t_l) \one$, $q^{t + 1} := \exp\left(\sum_{l = 1}^m w_l \ln q^t_l\right)$.
\end{lemma}

Now we turn to the proof of Theorem~\ref{thm:complexity_IBP}.
Next lemmas are preliminaries to the proof of correctness and complexity bound for Algorithm~\ref{Alg:dual_IBP}.

\begin{lemma}
\label{lemma:bounds_u_v}
    For any $t \ge 0$, $l = \overline{1, m}$ it holds
    \begin{align*}
        &\max_j [v_l^t]_j - \min_j [v_l^t]_j \le R_v,
        &\max_j [v_l^*]_j - \min_j [v_l^*]_j \le R_v,
    \end{align*}
    where 
    \begin{gather}
        R_v := \max_l \frac{\norm{C_l}_\infty}{\gamma} + \sum_{k = 1}^m w_k \frac{\norm{C_k}_\infty}{\gamma}.
    \end{gather}
\end{lemma}

\begin{proof}
    Bound can be derived in almost the same way as in~\cite{dvurechensky2018computational}. For $t = 0$ it obviously holds. Let us denote by $\nu_l$ the minimal entry of $K_l$:
    \[
    \nu_l := \min_{i,j} [K_l]_{i j} = e^{-\norm{C_l}_\infty / \gamma}.
    \]
    As $[K_l]_{i j} \le 1$, we obtain for all $j = \overline{1,n}$
    \[
    \ln \nu_l + \ln \la \one, e^{u_l} \ra
    \le [\ln K_l^T e^{u_l}]_j 
    \le \ln \la \one, e^{u_l} \ra,
    \]
    therefore
    \[
    \max_j [\ln K_l^T e^{u_l}]_j - \min_j [\ln K_l^T e^{u_l}]_j
    \le - \ln \nu_l
    = \frac{\norm{C_l}_\infty}{\gamma}.
    \]
    Hence at any update of $\v^t$ it holds
    \[
    \max_j [v_l^t]_j - \min_j  [v_l^t]_j
    \le \frac{\norm{C_l}_\infty}{\gamma} + \sum_{k = 1}^m w_k \frac{\norm{C_k}_\infty}{\gamma}
    \le R_v.
    \]
    For solution to the problem $(\u^\ast, \v^\ast)$ condition \eqref{eq:v_cond} also holds, and consequently the solution also meets derived bounds.
\end{proof}

Let us define an excess function 
\[\tilde{f}(\u, \v) := f(\u, \v) - f(\u^\ast, \v^\ast)\]
for further complexity analysis of Algorithm~\ref{Alg:dual_IBP}.

\begin{lemma}
\label{lemma:bound_excess}
    Let $\{\u^t, \v^t\}_{t \ge 0}$ be generated by Algorithm~\ref{Alg:dual_IBP}. Then for any even $t \ge 2$ we have
    \begin{equation}
        \tilde{f}(\u^t, \v^t) 
        \le R_v \sum_{l = 1}^m w_l \norm{q^t_l - \bar{q}^t}_1.
    \end{equation}
\end{lemma}

\begin{proof}
    Gradient inequality of any convex function $g$ at point $x^\ast$ reads as
    \[
    g(x^\ast) \ge g(x) + \la \nabla g(x), x^\ast - x \ra, \quad \forall x \in \operatorname{dom}(g).
    \]
    Applying the latter inequality to function $f$ at point $(\u^\ast, \v^\ast)$ we obtain
    \begin{align*}
        \tilde{f}(\u^t, \v^t)
        = f(\u^t, \v^t) - f(\u^\ast, \v^\ast) 
        \le \sum_{l = 1}^m w_l \la u^t_l - u^\ast_l, B_l(u^t_l, v^t_l) \one - p_l \ra + \sum_{l = 1}^m w_l \la v^t_l - v_l^\ast, q^t_l \ra.
    \end{align*}
    
    If $t \ge 2$ is even, then the first term in r.h.s.\ vanishes. Notice that $\la q^t_l, \one \ra = \la \one, B_l(u^t_l, v^t_l) \one \ra = \la \one, p_l \ra = 1$, thus
    \begin{align*}
        \tilde{f}(\u^t, \v^t)
        & = \sum_{l = 1}^m w_l \la v^t_l - v_l^*, q^t_l \ra
        = \sum_{l = 1}^m w_l \la v^t_l - v^*_l, q^t_l - \bar{q}^t \ra \\
        & = \sum_{l = 1}^m w_l \la (v^t_l - b^t_l \one) - (v^*_l - b^*_l \one), q^t_l - \bar{q}^t \ra \\
        & \le \sum_{l = 1}^m w_l \left(\norm{v^t_l - b^t_l \one}_\infty + \norm{v^*_l - b^*_l \one}_\infty\right) \norm{q^t_l - \bar{q}^t}_1,
    \end{align*}
    where 
    \[
    b_l^t := \frac{\min_i [v_l^t]_i + \max_i [v_l^t]_i}{2}, \quad
    b_l^* := \frac{\min_i [v_l^*]_i + \max_i [v_l^*]_i}{2}.
    \]
    By Lemma~\ref{lemma:bounds_u_v} $\norm{v^t_l - b^t_l \one}_\infty \le R_v / 2$ and $\norm{v^*_l - b^*_l \one}_\infty \le R_v / 2$, therefore
    \[
    \tilde{f}(\u^t, \v^t)
    \le R_v \sum_{l = 1}^m w_l \norm{q^t_l - \bar{q}^t}_1.
    \]
\end{proof}

\begin{lemma}
\label{lemma:IBP_decrease}
    For any odd $t \ge 1$ the following bound on the change of objective function $f(\cdot, \cdot)$ holds:
    \[
    f(\u^t, \v^t) - f(\u^{t+1}, \v^{t+1})
    \ge \frac{1}{11} \sum_{l = 1}^m w_l \norm{q_l^t - \bar{q}^t}_1^2,
    \]
    where $\bar{q}^t := \sum_{l = 1}^m w_l q^t_l$.
\end{lemma}

\begin{proof}
    If $t$ is odd, then $\v^{t+1}$ satisfies~\eqref{eq:v_cond} and $\u^t = \u^{t+1}$. Therefore,
    \begin{align*}
    	f(\u^t, \v^t) - f(\u^{t+1}, \v^{t+1}) 
    	& = \sum_{l = 1}^m w_l \bigl\{\la q_l^t, \one \ra - \la q_l^{t + 1}, \one \ra\bigr\} 
    	= \left\la \bar{q}^t - q^{t + 1}, \one \right\ra \\
    	& = \left\la \bar{q}^t - \exp\left(\sum\limits_{l = 1}^m w_l q_l^t\right), \one \right\ra
    	\ge \frac{4}{11} \sum_{j = 1}^n \frac{1}{[\bar{q}^t]_j} \sum_{l = 1}^m w_l \bigl([q_l^t - \bar{q}^t]_j^-\bigr)^2 \\
    	& = \frac{4}{11} \sum_{l = 1}^m w_l \sum_{j = 1}^n \frac{\bigl([q_l^t - \bar{q}^t]_j^-\bigr)^2}{[\bar{q}^t]_j}
    	\ge \frac{4}{11} \sum_{l = 1}^m w_l \frac{\left(\sum_j [q_l^t - \bar{q}^t]_j^-\right)^2}{\sum_{j = 1}^n [\bar{q}^t]_j} \\
    	& = \frac{1}{11} \sum_{l = 1}^m w_l \norm{q_l^t - \bar{q}^t}_1^2.
    \end{align*}
    Here we used equations $\la q_l^t, \one \ra = \la \one, p_l \ra = 1$, i.e.\ $q_l^t \in S_n(1)$ and thus $\bar{q}^t \in S_n(1)$, and the following fact: if $x \in \R_+^m$, $\bar{x} := \sum_{l = 1}^m w_l x_l$, then
    \[
    \bar{x} - \prod_{l = 1}^m x_l^{w_l} 
    \ge \frac{4}{11} \sum_{l = 1}^m w_l \frac{\bigl[(x_l - \bar{x})^-\bigr]^2}{\bar{x}}.
    \]
    Indeed, let $\Delta_l := x_l - \bar{x}$, then
    \[
    \bar{x} - \prod_{l = 1}^m x_l^{w_l} 
    = \bar{x} - \exp\left\{\sum_{l = 1}^m w_l \ln(\bar{x} + \Delta_l)\right\}
    = \bar{x} \left(1 - \exp\left\{\sum_{l = 1}^m w_l \ln\left(1 + \frac{\Delta_l}{\bar{x}}\right)\right\}\right),
    \]
    \[
    \sum_{l = 1}^m w_l \ln\left(1 + \frac{\Delta_l}{\bar{x}}\right)
    \le \sum_{l = 1}^m w_l \left(\frac{\Delta_l}{\bar{x}} - \frac{(\Delta_l^-)^2}{2 \bar{x}^2}\right)
    = -\sum_{l = 1}^m w_l \frac{(\Delta_l^-)^2}{2 \bar{x}^2}.
    \]
    Notice that $\Delta_l^- := \max\{- \Delta_l, 0\} = \max\{\bar{x} - x_l, 0\} \le \bar{x}$, thus 
    $\sum_{l = 1}^m w_l \frac{(\Delta_l^-)^2}{\bar{x}^2} \le 1$
    and
    \[
    \exp\left\{-\frac{1}{2} \sum_{l = 1}^m w_l \frac{(\Delta_l^-)^2}{\bar{x}^2}\right\}
    \le 1 - \left(1 - e^{-1/2}\right) \sum_{l = 1}^m w_l \frac{(\Delta_l^-)^2}{\bar{x}^2} 
    \le 1 - \frac{4}{11} \sum_{l = 1}^m w_l \frac{(\Delta_l^-)^2}{\bar{x}^2}.
    \]
    Consequently,
    \begin{equation*}
        \bar{x} - \prod_{l = 1}^m x_l^{w_l} 
        \ge \frac{4}{11} \sum_{l = 1}^m w_l \frac{(\Delta_l^-)^2}{\bar{x}}.
    \end{equation*}
\end{proof}

\begin{proof}[Proof of Theorem~\ref{thm:complexity_IBP}]
    First, notice that 
    \[
    H\bigl(B_l(u_l, v_l)\bigr) = \la u_l, B_l(u_l, v_l) \one \ra + \la v_l, B_l^T(u_l, v_l) \one \ra - \la \one, B_l(u_l, v_l) \one \ra - \frac{1}{\gamma} \la C_l, B_l(u_l, v_l) \ra.
    \]
    Since $N$ is even, $B_l \one = p_l$ and therefore it holds
    \begin{align*}
        \sum_{l = 1}^m w_l \left(\left\la C_l, B_l \right\ra + \gamma H(B_l)\right) 
        & = \gamma \sum_{l = 1}^m w_l \left(\la u^N_l, B_l \one \ra + \la v^N_l, B_l^T \one \ra - \la \one, B_l \one \ra \right) \\
        & = - \gamma \sum_{l = 1}^m w_l \left(\la \one, B_l \one \ra - \la u^N_l, p_l \ra\right) + \gamma \sum_{l = 1}^m w_l \la v^N_l, q_l^N \ra \\
        & = - \gamma f(\u^N, \v^N) + \gamma \sum_{l = 1}^m w_l \la v^N_l, q_l^N - \bar{q}^N \ra.
    \end{align*}
    Now Lemmas~\ref{lemma:dual_2}, \ref{lemma:bounds_u_v}, \ref{lemma:bound_excess}, and stopping criterion yield
    \begin{align*}
        \sum_{l = 1}^m w_l \left(\left\la C_l, B_l \right\ra + \gamma H(B_l)\right) 
        & - \sum_{l = 1}^m w_l \left(\left\la C_l, \pi_l^* \right\ra + \gamma H(\pi^*_{\gamma, l})\right) \\
        & = \gamma \bigl(f(\u^*, \v^*) - f(\u^N, \v^N)\bigr) + \gamma \sum_{l = 1}^m w_l \la v^N_l, q_l^N - \bar{q}^N \ra \\
        & \le \gamma \sum_{l = 1}^m w_l \frac{R_v}{2} \norm{q_l^N - \bar{q}^N}_1
        \le \frac{\gamma R_v}{2} \e'
        \le \max_l \norm{C_l}_\infty \e'.
    \end{align*}

    Now let us prove the complexity bound. We will do it in two steps.

    \paragraph{1.} If $t \ge 2$ is even, then by Lemma~\ref{lemma:bound_excess}
    \[
	\tilde{f}(\u^t, \v^t) \le R_v \sum_{l = 1}^m w_l \norm{q_l^t - \bar{q}^t}_1
    \]
    and since stopping criterion is not fulfilled, 
    \[
    \sum_{l = 1}^m w_l \norm{q_l^t - \bar{q}^t}_1 > \e'.
    \]
    Inequality $\sum_{l = 1}^m w_l \norm{q_l^t - \bar{q}^t}_1^2 \ge \left(\sum_{l = 1}^m w_l \norm{q_l^t - \bar{q}^t}_1\right)^2$ together with Lemma~\ref{lemma:IBP_decrease} give us the following bound:
    \begin{align*}
        \tilde{f}(\u^t, \v^t) - \tilde{f}(\u^{t+1}, \v^{t+1}) 
        & = f(\u^t, \v^t) - f(\u^{t+1}, \v^{t+1}) \\
        & \ge \frac{1}{11} \max\left\{(\e')^2, \left(\frac{\tilde{f}(\u^t, \v^t)}{R_v}\right)^2\right\} 
        = \max\left\{\frac{(\e')^2}{11}, \frac{1}{11 R_v^2} \tilde{f}^2(\u^t, \v^t)\right\}.
    \end{align*}
    If $t$ is odd then we have at least $\tilde{f}(\u^{t+1}, \v^{t+1}) \le \tilde{f}(\u^t, \v^t)$.
    To simplify derivation we define
    \begin{align}
         \delta_t := \tilde{f}(\u^t, \v^t).
    \end{align}
    
    Now we have two possibilities to estimate number of iteration. The first one is based on inequalities
    \[
    \frac{1}{\delta_{t + 1}} \ge 
    \begin{cases}
        \frac{1}{\delta_t} + \frac{1}{11 R_v^2} \frac{\delta_t}{\delta_{t+1}} \ge \frac{1}{\delta_t} + \frac{1}{11 R_v^2}, & t \bmod 2 = 0,\\
        \frac{1}{\delta_t}, & t \bmod 2 = 1.
    \end{cases}
    \]
    Summation of these inequalities gives
    \[
    \frac{1}{\delta_t} \ge \frac{1}{\delta_1} + \frac{t - 2}{22 R_v^2}
    \]
    and hence
    \begin{equation}
    \label{eq:self_est}
        t \le 2 + 22 R_v^2 \left(\frac{1}{\delta_t} - \frac{1}{\delta_1}\right).
    \end{equation}
    
    The second estimate can be obtain from
    \[
    \delta_{t+1} \le 
    \begin{cases}
        \delta_t - \frac{(\e')^2}{11}, & t \bmod 2 = 0,\\
        \delta_t, & t \bmod 2 = 1.
    \end{cases}
    \]
    Similarly, summation of these inequalities gives
    \begin{align}
    \label{eq:eps_est}
        & \delta_t \ge \delta_t - \delta_{t + k} 
        \ge \frac{k - 1}{22} (\e')^2, 
        \\
        & k \le 1 + \frac{22 \delta_t}{(\e')^2}.
    \end{align}
    
    \paragraph{2.} To combine the two estimates~\eqref{eq:self_est} and~\eqref{eq:eps_est}, we consider a switching strategy parametrized by number $s \in (0, \delta_1)$. First $t$ iterations we use~\eqref{eq:self_est}, resulting in $\delta_t$ becomes below some $s$. Then, we use $s$ as a starting point and estimate the remaining number of iteration by~\eqref{eq:eps_est}. The quantity $s$ can be found from the minimization
    \[
    N = t + k 
    \le 4 + \frac{2 s}{(\e')^2} + 22 R_v^2 \left(\frac{1}{s} - \frac{1}{\delta_1}\right).
    \]
    Minimizing the r.h.s.\ of the latter inequality in $s$ leads to
    \begin{equation}
    \label{case1}
        N \le \min_{0 \le s \le \delta_1} \left\{4 + \frac{22 s}{(\e')^2} + 22 R_v^2 \left(\frac{1}{s} - \frac{1}{\delta_1}\right)\right\} 
        \le 4 + \frac{44 R_v}{\e'}.
    \end{equation}
    The last inequality is obtained by the substitution $s = R_v \e'$ that is the solution to the minimization problem.
    Of course, the switching strategy is impossible if $\delta_1 < s$. But in this case~\eqref{eq:eps_est} gives
    \begin{equation}
    \label{case2}
        N \le 2 + \frac{22 \delta_1}{(\e')^2} < 4 + \frac{44 R_v}{\e'}.
    \end{equation}
    In both cases~\eqref{case1} and~\eqref{case2} we have
    \[
    N \le 4 + \frac{44 R_v}{\e'} 
    = O\left(\frac{\max_l \norm{C_l}_\infty}{\gamma \e'}\right).
    \]
\end{proof}


\subsection{Approximating Non-regularized WB by IBP}

To find an approximate solution to the initial problem~\eqref{prob:unreg_bary} we apply Algorithm~\ref{Alg:dual_IBP} with a suitable choice of $\gamma$ and $\e'$ 
and average marginals $q_1, \dots, q_m$ with weights $w_l$, what leads to Algorithm~\ref{Alg:IBP_to_bary}.

\begin{algorithm}[H]
    \caption{Finding Wasserstein barycenter by IBP}
    \label{Alg:IBP_to_bary}
    \begin{algorithmic}[1]
        \REQUIRE {Accuracy $\e$; cost matrices $C_1, \dots, C_m$; marginals $p_1, \dots, p_m$}
        \STATE {Set  $\gamma := \frac{1}{4} \frac{\e}{\ln n}, \quad \e' :=  \frac{1}{4} \frac{\e}{\max_l \norm{C_l}_\infty}$}
        \STATE {Find $B_1 := B_1(u^t_1, v^t_1), \dots, B_m := B_m(u^t_m, v^t_m)$ by Algorithm~\ref{Alg:dual_IBP} with accuracy $\e'$}
        \STATE {$q := \tfrac{1}{\sum_{l = 1}^m w_l \la \one, B_l \one \ra} \sum_{l = 1}^m w_l B_l^T \one$}
        \ENSURE $q$
    \end{algorithmic}
\end{algorithm}

To present our final complexity bound for Algorithm~\ref{Alg:IBP_to_bary}, which calculates approximated non-regularized Wasserstein barycenter, we formulate the following auxiliary Algorithm~\ref{Alg:round} finding a projection to the feasible set. It is based on Algorithm~2 from \cite{altschuler2017near-linear}. Notice that the bound~\eqref{eq:proj_bound} follows immediately from the proof of \cite[Theorem~4]{altschuler2017near-linear}, although it was not stated explicitly.
\begin{algorithm}[H]
    \caption{Round to feasible solution}
    \label{Alg:round}
    \begin{algorithmic}[1]
        \REQUIRE {$B_1, \dots, B_m \in \R_+^{n \times n}$, $p_1, \dots, p_m \in S_n(1)$} 
        \STATE {$q := \tfrac{1}{\sum_{l = 1}^m w_l \la \one, B_l \one \ra} \sum_{l = 1}^m w_l B_l^T \one$}
        \STATE {Calculate $\check{B}_1, \dots, \check{B}_m$ by Algorithm~2 from~\cite{altschuler2017near-linear} s.t.\\
        \begin{equation}\label{eq:proj_bound}
            \check{B}_l \in \Pi(p_l, q), \quad
            \norm{\check{B}_l - B_l}_1 \le 2 \left(\sum_i [B_l \one - p_l]_i^+ + \sum_j [B_l^T \one - q]_j^+\right)
        \end{equation}}
        \ENSURE {$\check{B}_1, \dots, \check{B}_m$}
    \end{algorithmic}
\end{algorithm}

Next theorem presents complexity bound for Algorithm~\ref{Alg:IBP_to_bary}.
\begin{theorem}
\label{thm:complexity_IBP_to_bary}
    For given $\e$ Algorithm~\ref{Alg:IBP_to_bary} returns $q \in S_n(1)$ s.t.
    \[
    \sum_{l = 1}^m w_l \W(p_l, q) 
    - \sum_{l = 1}^m w_l \W(p_l, q^*) \le \e,
    \]
    where $q^*$ is a solution to non-regularized problem~\eqref{prob:unreg_bary}.
    It requires
    \[
    O\left(\left(\frac{\max_l \norm{C_l}_\infty}{\e}\right)^2 M_{m,n} \ln n + m n\right)
    \]
    arithmetic operations, where $M_{m,n}$ is a time complexity of one iteration of Algorithm~\ref{Alg:dual_IBP}.
\end{theorem}

\begin{Rem}
    As each iteration of Algorithm~\ref{Alg:dual_IBP} requires $m$ matrix-vector multiplications, the general bound is $M_{m,n} = O(m n^2)$. However, for some specific form of matrices $C_l$ it's possible to achieve better complexity, e.g.\ $M_{m,n} = O(m n \log n)$ via FFT \cite{peyre2018computational}.
\end{Rem}

\begin{proof}
    Let $\Bpi^*_\gamma = [\pi^*_{\gamma,1}, \dots, \pi^*_{\gamma,m}]$ be a solution to~\eqref{prob:reg_bary_2} and $\Bpi^* = [\pi_1^*, \dots, \pi_m^*]$ be a solution to the non-regularized problem. Consider $\check{B}_1, \dots, \check{B}_m$ obtained from $B_1(u^t_1, v^t_1), \dots, B_m(u^t_m, v^t_m)$ via Algorithm~\ref{Alg:round}, where $t$ is the number of IBP iterations. Then we obtain
    \[
    \sum_{l = 1}^m w_l \la C_l, \check{B}_l \ra 
    \le \sum_{l = 1}^m w_l \left(\la C_l, B_l(u^t_l, v^t_l) \ra + \norm{C_l}_\infty \norm{B_l(u^t_l, v^t_l) - \check{B}_l}_1\right).
    \]
    By Theorem~\ref{thm:complexity_IBP} one obtains
    \begin{align*}
        \sum_{l = 1}^m w_l \la C_l, B_l(u^t_l, v^t_l) \ra
        & \le \sum_{l = 1}^m w_l \left(\la C_l, \pi^*_{\gamma,l} \ra + \gamma H(\pi^*_{\gamma,l}) - \gamma H(B_l(u^t_l, v^t_l))\right) + \max_l \norm{C_l}_\infty \e' \\
        & \le \sum_{l = 1}^m w_l \left(\la C_l, \pi_l^* \ra + \gamma H(\pi^*_l) - \gamma H(B_l(u^t_l, v^t_l))\right) + \max_l \norm{C_l}_\infty \e' \\
        & \le \sum_{l = 1}^m w_l \W(p_l, q^*) + 2 \gamma \ln n + \max_l \norm{C_l}_\infty \e'.
    \end{align*}
    Here we used inequalities $- 2 \ln n \le H(\pi) + 1 \le 0$ holding on $S_{n \times n}(1)$.
    As stopping time $t$ is even, $B_l(u^t_l, v^t_l) \one = p_l$ and $\la q_l^t, \one \ra = 1$, therefore
    \[
    \norm{B_l(u^t_l, v^t_l) - \check{B}_l}_1
    \le 2 \sum_j [q_l^t - \bar{q}^t]_j^+
    = \norm{q_l^t - \bar{q}^t}_1,
    \]
    and hence
    \[
    \sum_{l = 1}^m w_l \norm{B_l(u^t_l, v^t_l) - \check{B}_l}_1 \le \e'.
    \]
    Notice that $\check{B}_l \in \Pi(p_l, q)$ for all $1 \le l \le m$, consequently
    \begin{align*}
        \sum_{l = 1}^m w_l \W(p_l, q) 
        & \le \sum_{l = 1}^m w_l \la C_l, \check{B}_l \ra \\ 
        & \le \sum_{l = 1}^m w_l \la C_l, B_l(u^t_l, v^t_l) \ra + \max_l \norm{C_l}_\infty \sum_{l = 1}^m w_l \norm{B_l(u^t_l, v^t_l) - \check{B}_l}_1\\
        & \le \sum_{l = 1}^m w_l \W(p_l, q^*) + 2 \gamma \ln n + 2 \max_l \norm{C_l}_\infty \e'
        \le \sum_{l = 1}^m w_l \W(p_l, q^*) + \e.
    \end{align*}
    
    Complexity bound for the algorithm is a simple corollary of Theorem~\ref{thm:complexity_IBP}.
\end{proof}

\begin{Rem}
    Notice that according to the proof of the above theorem, one can also reconstruct approximated optimal transportation plans $\check{B}_l$ between $p_l$ and approximated barycenter $q$ using Algorithm~\ref{Alg:round}.
\end{Rem}